\documentclass[10pt]{article}
\usepackage{cite}
\usepackage{hyperref}
\usepackage[alphabetic]{amsrefs}
\usepackage{amssymb}
\usepackage{amsfonts}
\usepackage{amsthm}
\usepackage{amsmath}
\usepackage{mathtools}
\newtheorem{Theorem}{Theorem}

\newtheorem{Lemma}{Lemma}

\newtheorem{Corollary}{Corollary}
\usepackage{hieroglf}

\usepackage{fullpage}
\usepackage{algorithm}
\usepackage{algorithmic}
\usepackage{tikz}
\usepackage{wasysym}
\usetikzlibrary{matrix}

\newcommand{\rank}{\text{rank}}
\newcommand{\Spec}{\text{Spec}}
\newcommand{\dom}{\trianglelefteq}

\newcommand{\Span}{\text{span}}

\title{General Linear and Symplectic Nilpotent Orbit Varieties}
\author{Samuel Reid\thanks{This research was supported by an Undergraduate Student Research Award of the Natural
Sciences and Engineering Council of Canada. Special thanks are to my supervisor Clifton Cunningham for many insightful discussions.}}

\begin{document}

\maketitle

\begin{abstract}
The condition of nilpotency is studied in the general linear Lie algebra $\mathfrak{gl}_{n}(\mathbb{K})$ and the
symplectic Lie algebra $\mathfrak{sp}_{2m}(\mathbb{K})$ over an algebraically closed field of characteristic 0.
In particular, the conjugacy class of nilpotent matrices is described through
nilpotent orbit varieties $\mathcal{O}_{\lambda}$ and an algorithm is provided for computing the closure $\overline{\mathcal{O}_{\lambda}} \cong \Spec\left(\mathbb{K}[X]\big/J_{\lambda}\right).$
We provide new generators for the ideal $J_{\lambda}$ defining the affine variety $\overline{\mathcal{O}_{\lambda}}$
which show that the generators provided in \cite{Weyman} are not minimal. Furthermore, we conjecture the existence
of local weak N\'{e}ron models for nilpotent orbit varieties based on bounding $p$ in the polynomial ring
with p-adic integer coefficients for which the equations defining $\mathcal{O}_{\lambda}$ can embed.
\end{abstract}

\section{Introduction}
Let $\mathbb{K}$ be an algebraically closed field of characteristic zero. We are interested in geometrically describing the condition of nilpotency in the general linear Lie algebra $\mathfrak{gl}_{n}(\mathbb{K})$ through associating varieties with conjugacy classes of nilpotent elements in $\mathfrak{gl}_{n}(\mathbb{K})$. Let $X$ be an $n \times n$ matrix in the nilpotent cone or nullcone $\mathcal{N}(n):= \mathfrak{gl}_{n}^{\text{nilp}}(\mathbb{K}) = \{X \in \mathfrak{gl}_{n}(\mathbb{K}) \; | \; X^{k} = 0, \exists k \in \mathbb{N}\}$, and denote the conjugacy class (similarity class) of $X$, i.e., the orbit of $X$ under the action of conjugation, by $\mathcal{C}_{X} = \{P^{-1} X P \; | \; P \in \mathfrak{gl}_{n}(\mathbb{K})\}$. We denote the origin of the nilpotent cone by $\mathcal{N}_{0}(n):= \{x_{ij} = 0 \; | \; 1 \leq i,j \leq n\}$. By the Jordan normal form theorem, $\exists P \in \mathfrak{gl}_{n}(\mathbb{K})$ so that $Y = P^{-1} X P$ has Jordan blocks of sizes determined by an integer partition $\lambda_{Y} = [\lambda_{1},...,\lambda_{l}]$ of $n$ with $\lambda_{1} \geq \cdot\cdot\cdot \geq \lambda_{l}$. Thus, the map $\mathcal{C}_{X} \mapsto \lambda_{Y}$ is a bijection between the set of nilpotent conjugacy classes and the set of partitions of $n$. Letting $\lambda = [\lambda_{1},...,\lambda_{l}]$ and $\lambda' = [\lambda_{1}',...,\lambda_{s}']$ be partitions of the integer $n$ listed in a non-increasing sequence, the dominance order $\dom$ on the set of partitions of a positive integer $n$ is defined by $\lambda \dom \lambda'$ if $\sum_{i=1}^{k} \lambda_{i} \leq \sum_{i=1}^{k} \lambda_{i}'$ for all $k \leq \max\{l,s\}$. If $l > s$ then we add $l-s$ zeros to end of the partition $\lambda'$ and if $s>l$ then we add $s-l$ zeros to end of the partition $\lambda$ for this definition to be well-defined. Through this bijection, the dominance ordering of integer partitions partially orders the set of nilpotent conjugacy classes. The nilpotent orbit variety $\mathcal{O}_{\lambda}$ associated with the nilpotent conjugacy class in bijection with the partition $\lambda$ is shown to be given by exact conditions on ranks of powers of matrices, where $\mathbb{K}[X]:=\mathbb{K}[x_{ij} \; | \; 1 \leq i,j \leq n]$. Thus,
$$\mathcal{O}_{\lambda} = \{f \in \mathbb{K}[X] \; | \; \rank(X^{k}) = r, \forall (k,r) \in U_{\lambda}\}$$
with $$U_{\lambda} = \left\{(k,r) \; | \; 1 \leq k \leq \max\{\lambda_{1},...,\lambda_{l}\}, r = \sum_{i=1}^{l} f^{k}(\lambda_{i})\right\}$$ and the rank counting function $f$ defined by $$f(x)=
\begin{cases}
x-1 \;\;\;\; \text{if} \; x > 0 \\
0 \;\;\;\;\;\;\;\;\;\; \text{if} \; x \leq 0
\end{cases}$$
We remark that $\rank(X^{k}) = \sum_{i=1}^{l} f^{k}(\lambda_{i})$ because the entries in the first upper diagonal of $X$ in Jordan normal form pass to the second upper diagonal of $X^2$ and so on until the nilpotency of $X$ ends this marching of the entries away from the main diagonal. From this observation, the non-zero entries in the Jordan blocks of $X$ are then naturally kept track of by powers of the rank counting function.

We have that the Zariski closure of a nilpotent orbit variety $\overline{\mathcal{O}_{\lambda}}$ associated with the nilpotent conjugacy class in bijection with the partition $\lambda$ is defined by upper bounds on ranks of powers of matrices. Thus,
$$\overline{\mathcal{O}_{\lambda}} = \{f \in \mathbb{K}[X] \; | \; \rank(X^{k}) \leq r, \forall (k,r) \in U_{\lambda}\}$$
Using the dominance ordering of integer partitions and thus nilpotent orbit varieties, we express the closure of a nilpotent orbit variety in terms of nilpotent orbit varieties by
$$\overline{\mathcal{O}_{\lambda}} = \mathcal{O}_{\lambda} \cup \left(\bigcup_{\mu \triangleleft \lambda} \mathcal{O}_{\mu}\right).$$
We can visualize the nilpotent cone $\mathcal{N}(n)$ as the union of all nilpotent orbit varieties as seen in Figure \ref{nullcone}.

 \begin{figure}[h!]\label{nullcone}
 \begin{center}
 \includegraphics[scale=0.3]{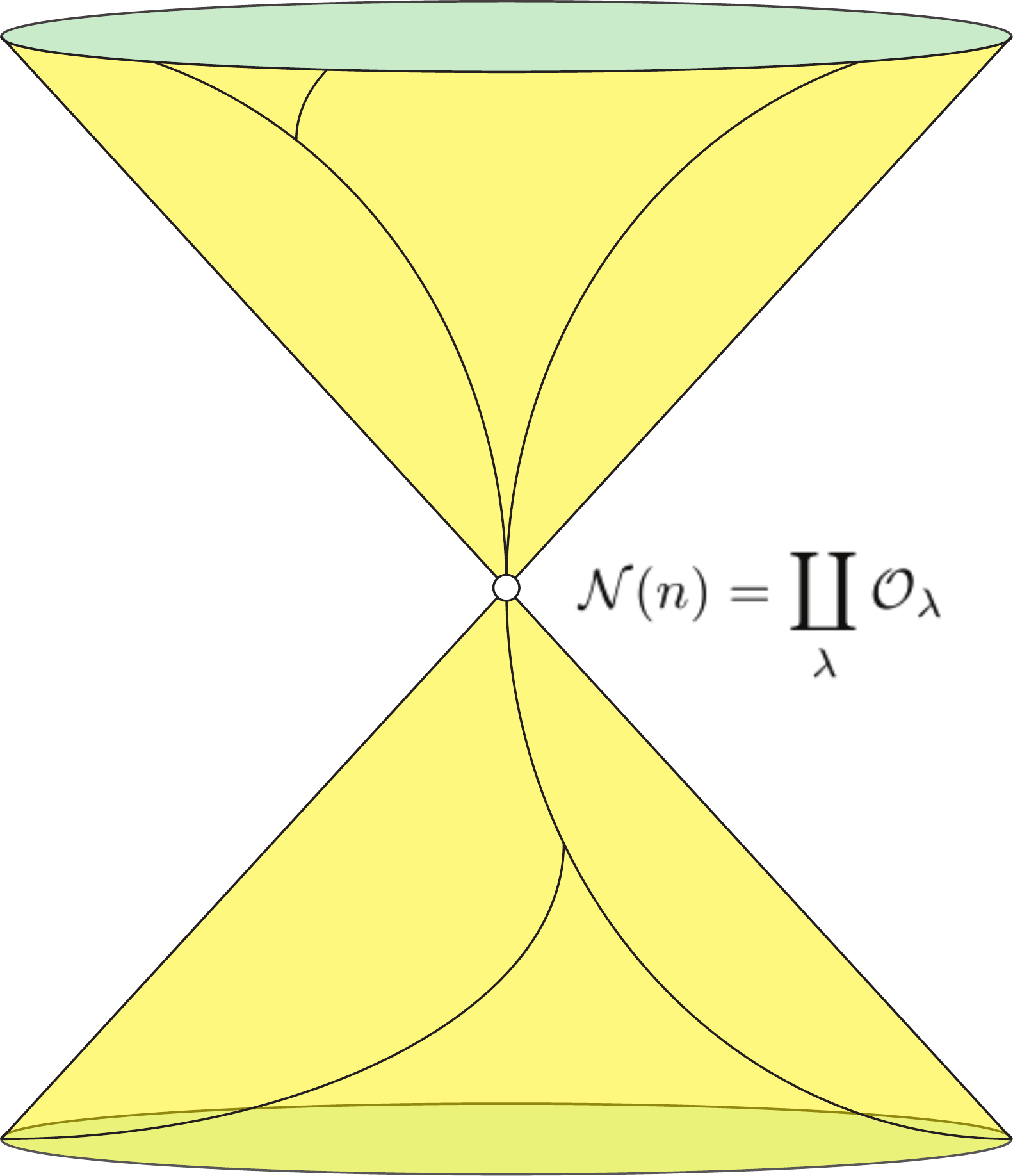}
 \end{center}
 \caption{A representation of the nilpotent cone with each region denoting a unique nilpotent orbit variety.}
 \end{figure}

\section{Nilpotent Orbit Varieties and Ideal Generators}
Since $\overline{\mathcal{O}_{\lambda}}$ is an affine variety, it is defined by an ideal $J_{\lambda}$ associated with the partition $\lambda$ by
$\overline{\mathcal{O}_{\lambda}} \cong \Spec\left(\mathbb{K}[X]\big/J_{\lambda}\right).$
We use a more recent rephrasing of Theorem 4.6 of \cite{Weyman} given by Theorem 5.4.3 of \cite{KLMW} regarding the generators of $J_{\lambda}$ and state

\begin{Theorem}\label{generatorthm}
The ideal $J_{\lambda}$ is generated by $V_{0,p}(1\leq p \leq n)$ and $V_{i,\lambda(i)}(1 \leq i \leq n)$ where $\lambda(i) = \lambda_{1} + \cdot\cdot\cdot + \lambda_{i} - i + 1$ and $V_{i,p}$ is defined as a span of linear combinations
$$V_{i,p} := \Span\left\{\sum_{|J|=p-i} X(P,J|Q,J) \; \bigg| \; P,Q \subset \{1,...,n\}, |P|=|Q|=i, (P \cup Q) \cap J = \emptyset \right\}$$
where $X(P|Q)$ denotes the minor of $X \in \mathfrak{gl}_{n}(\mathbb{K})$ with rows indexed by $P$ and columns indexed by $Q$.
\end{Theorem}
\begin{proof}
This is a restatement of Theorem 4.6 of \cite{Weyman} using the alternative definition of $V_{i,p}$ given on page 30 of \cite{KLMW}. In \cite{Weyman},
$$V_{i,p} \cong \bigwedge^{i}V^{*} \bigotimes \bigwedge^{i} V$$ with elements given from a basis $e_{1},...,e_{n}$ of the vector space $V$ by $e_{p_{1}}^{*} \wedge e_{p_{2}}^{*} \wedge \cdot\cdot\cdot \wedge e_{p_i}^{*} \otimes e_{q_{1}} \wedge e_{q_{2}} \wedge \cdot\cdot\cdot \wedge e_{q_{i}}$. Whereas, in \cite{KLMW}
$$V_{i,p} = \Span\left\{\sum_{|J|=p-i} X(P,J|Q,J) \; \bigg| \; P,Q,J \subset \{1,...,n\}, |P|=|Q|=i, (P \cup Q) \cap J = \emptyset \right\}$$
with the proof that $J_{\lambda}$ is generated by $V_{0,p}\text{(}1\leq p \leq n\text{)}$ and $V_{i,\lambda(i)}\text{(}1 \leq i \leq n\text{)}$ given in \cite{Weyman} using Lascoux resolution of complexes, Schur functors used to define irreducible representations of $\mathfrak{gl}_{n}$, spectral sequences of filtrations, and induction on the length of the partition.
\end{proof}

In order to recover the nilpotent orbit variety $\mathcal{O}_{\lambda}$ from the closure $\overline{\mathcal{O}_{\lambda}}$, we construct the set $$H_{\lambda} = \{h \in \mathbb{K}[X] \; | \; \rank(X^{k}) \geq r, \forall (k,r) \in U_{\lambda}\}$$ and use localization. Since $\rank(X^{k}) \geq r$ is guaranteed by the existence of an $r \times r$ minor of $X$ with non-zero determinant, we construct another set $$H_{\lambda}^{k} = \left\{X(P|Q) \neq 0 \; \big| \; P,Q \subseteq \{1,...,n\},|P|=|Q|=r_{k}:=\rank(X^k)=\sum_{i=1}^{l} f^{k}(\lambda_{i})\right\}$$
which indexes the $r_{k} \times r_{k}$ minors of $X^{k}$. Then since there are ${n \choose r}^2$ minors of $X$ with size $r \times r$,
$$H_{\lambda} = \bigcup_{k=1}^{\max\{\lambda\}} H_{\lambda}^{k} = \left\{h_{j,k} \in H_{\lambda}^{k} \; \big| \; 1 \leq k \leq \max\{\lambda\},1 \leq j \leq {n \choose r_{k}}^{2}\right\}$$
where $\max\{\lambda\} = \max\{\lambda_{1},...,\lambda_{l}\}$. We now take unions of localizations of nilpotent orbit variety closures by $h_{j,k} \in H_{\lambda}$ and obtain
\begin{align*}
\mathcal{O}_{\lambda} &= \bigcup_{h \in H_{\lambda}} \left(\overline{\mathcal{O}_{\lambda}}\right)_{h} = \bigcup_{k=1}^{\max\{\lambda\}} \bigcup_{j=1}^{{n \choose r_{k}}^2} \left(\overline{\mathcal{O}_{\lambda}}\right)_{h_{j,k}} \cong \bigcup_{k=1}^{\max\{\lambda\}} \bigcup_{j=1}^{{n \choose r_{k}}^2} \left(\Spec\left(\left(\mathbb{K}[X]\big/J_{\lambda}\right)\right)\right)_{h_{j,k}} \\
&\cong \bigcup_{k=1}^{\max\{\lambda\}} \bigcup_{j=1}^{{n \choose r_{k}}^2} \Spec\left(\left(\mathbb{K}[X]\big/J_{\lambda}\right)_{h_{j,k}}\right) \cong \bigcup_{k=1}^{\max\{\lambda\}} \bigcup_{j=1}^{{n \choose r_{k}}^2} \Spec\left(\frac{\mathbb{K}[X,t]}{J_{\lambda}\langle h_{j,k} t - 1 \rangle}\right)
\end{align*}
where $\left(\cdot\right)_{h}$ denotes localization at $h$. We remark that the transition maps for this atlas are induced by the isomorphism $$\left(\Spec\left(\mathbb{K}[X]\big/J_{\lambda}\right)\right)_{hh'} \cong \left(\Spec\left(\mathbb{K}[X]\big/J_{\lambda}\right)\right)_{h'h}$$
where $h,h' \in H_{\lambda}$.

\section{Computing Nilpotent Orbits in $\mathfrak{gl}_{n}$}
To gain some intuition for what $V_{i,p}$ represents in the formulation in \cite{Weyman} and in \cite{KLMW} we present an example which illustrates both. We first remark that the condition that $(P \cup Q) \cap J = \emptyset$ ensures that the minor $X(P,J | Q,J)$ is square and thus has a well-defined determinant. With this in mind, we compute the nilpotent orbit variety $\mathcal{O}_{[2,1]}$ in $\mathcal{N}(3):= \mathfrak{gl}_{3}^{\text{nilp}}(\mathbb{K})$ using a simple construction which yields generators for $J_{[2,1]}$ which are more minimal than in Theorem \ref{generatorthm} before presenting this case in the harder to understand language of $V_{i,p}$'s. We conjecture that for small values of $n$ the generators presented in our algorithm are less minimal than those constructed by Weyman.

We begin with the bijection between integer partitions and nilpotent orbit varieties,
$$
[2,1] \mapsto \mathcal{O}_{[2,1]} = \left\{f \in \mathbb{K}[X] \; | \; \rank(X) = 1, X^2 =0 \right\} \ni
\left[
\begin{array}{c c c}
0 & 1 & 0 \\
0 & 0 & 0 \\
0 & 0 & 0
\end{array} \right]$$
where $U_{[2,1]} = \{(1,1),(2,0)\}$. We now compute the nilpotent orbit variety closure $\overline{\mathcal{O}_{[2,1]}}$ by using a lemma which upper bounds the rank of a matrix by conditions on the determinants of minors of the matrix.

\begin{Lemma}\label{minorslemma}
If $X \in \mathfrak{gl}_{n}$ and $\det(M) = 0$ for every $(r+1)\times(r+1)$ minor $M$ of $X$, then $\rank(X) \leq r$. That is, if $X(P,Q) = 0$ for every $P,Q \subset \{1,...,n\}$ with $|P|=|Q|=r+1$, then $\rank(X) \leq r$.
\end{Lemma}
\begin{proof}
The rank of a matrix can be equivalently defined as the dimension of the largest minor whose determinant is not zero. Hence, if the determinant of every $(r+1)\times(r+1)$ minor of $X$ is zero then $\rank(X) \leq r$.
\end{proof}

From computing nilpotent orbit variety closures we can recover the nilpotent orbit variety in this case by using
$$\overline{\mathcal{O}_{[2,1]}} = \mathcal{O}_{[2,1]} \cup \mathcal{O}_{[1,1,1]}$$
since $\mathcal{O}_{[1,1,1]} = \{f \in \mathbb{K}[X] \; | \; X = 0\} = \{x_{11} =0,...,x_{33}=0\} = \mathcal{N}_{0}(3)$. Now, $$\overline{\mathcal{O}_{[2,1]}} = \left\{f \in \mathbb{K}[X] \; | \; \rank(X) \leq 1, X^2 =0 \right\}$$
we have that $\rank(X) \leq 1$ is satisfied when every $2 \times 2$ minor of $X$ has determinant zero and that $X^2 = 0$ is satisfied when every $1 \times 1$ minor of $X^2$ has determinant zero, that is, when each entry of $X^2$ is zero. Thus,
\begin{align*}
\overline{\mathcal{O}_{[2,1]}} = \{
&x_{12}x_{33} - x_{13}x_{32},x_{11}x_{32}-x_{12}x_{31},x_{11}x_{22}-x_{12}x_{21},x_{12}x_{23}-x_{13}x_{22},x_{21}x_{32} - x_{22}x_{31}, \\
&x_{22}x_{33}-x_{23}x_{32},x_{11}x_{23}-x_{13}x_{21},x_{21}x_{33}-x_{23}x_{31},x_{11}x_{33}-x_{13}x_{31}, x_{11}^2 + x_{12}x_{21} + x_{13}x_{31}, \\
&x_{11}x_{12} + x_{12}x_{22} + x_{13}x_{33},x_{21}x_{11} + x_{22}x_{21} + x_{23}x_{31}, x_{21}x_{11} + x_{22}x_{21} + x_{23}x_{31}, \\ &x_{21}x_{12} + x_{22}^2 + x_{23}x_{32}, x_{21}x_{13} + x_{22}x_{23}+x_{23}x_{33}, x_{31}x_{11}+x_{32}x_{21}+x_{33}x_{31}, \\
&x_{31}x_{12}+x_{32}x_{22}+x_{33}x_{32},x_{31}x_{13}+x_{32}x_{23}+x_{33}^2\}
\end{align*}
which is a system of 18 polynomial equations in $\mathbb{K}[x_{11},x_{12},x_{13},x_{21},x_{22},x_{23},x_{31},x_{32},x_{33}]$. We then have that $\mathcal{O}_{[2,1]} = \overline{\mathcal{O}_{[2,1]}} \backslash \mathcal{N}_{0}(3)$, where $\mathcal{N}_{0}(3)$ denotes the origin of the nilpotent cone in $\mathfrak{gl}_{3}$.
In general, we refer to Algorithm 1 for computing nilpotent orbit variety closures in terms of $\lambda$.

\begin{algorithm}
\caption{$\mathfrak{gl}_{n}$ Nilpotent Orbit Variety Closure}
\begin{algorithmic}
\REQUIRE $\lambda = [\lambda_{1},...,\lambda_{l}]$, where $\sum_{i=1}^{l} \lambda_{i} = n$ and $\lambda_{i} \in \mathbb{N}$, $\forall i \in \{1,...,l\}$.
\STATE Set $\overline{\mathcal{O}_{\lambda}} = \emptyset$.
\FORALL{$k \in \{1,...,n\}$}
\STATE Set $r = \rank(X^{k}) = \sum_{i=1}^{l} f^{k}(\lambda_{i})$
\IF{$r \geq 0$}
\FORALL{$P,Q \subset \{1,...,n\}$}
\IF{$|P|=|Q|=r+1$}
\STATE Set $\overline{\mathcal{O}_{\lambda}} = \overline{\mathcal{O}_{\lambda}} \cup \{X^{k}(P|Q) = 0\}$.
\ENDIF
\ENDFOR
\ENDIF
\ENDFOR
\RETURN $\overline{\mathcal{O}_{\lambda}}$
\end{algorithmic}
\end{algorithm}

\pagebreak

In the formalism presented by Weyman we have that
$$\overline{\mathcal{O}_{[2,1]}} \cong \Spec\left(\mathbb{K}[X]\big/J_{[2,1]}\right)$$
where $J_{[2,1]} = \langle V_{0,1},V_{0,2},V_{0,3},V_{1,2},V_{2,2},V_{3,1}\rangle$, which as we will see reduces to $\langle V_{0,1},V_{0,2},V_{0,3},V_{1,2},V_{2,2}\rangle$ since $V_{i,p}$ is trivial for $i > p$. The function $\lambda(i) = \lambda_{1} + \cdot\cdot\cdot \lambda_{i} - i + 1$ is used to apply Theorem \ref{generatorthm} to this example as follows. For the partition $\lambda = [2,1]$, we append $i - |\lambda|$ additional zeroes if required to define $V_{i,p}$ for a specific $p = \lambda(i)$. In this case we have $\lambda(1)=2, \lambda(2)=2$, and $\lambda(3) = 1$ are the values of $p$ for each non-zero $i$. Then,
\begin{align*}
V_{0,1} &= \Span\left\{\sum_{|J|=1} X(J|J) \; \big| \; J \subset \{1,2,3\} \right\} = \{x_{11} + x_{22} + x_{33}\} \\
V_{0,2} &= \Span\left\{\sum_{|J|=2} X(J|J) \; \big| \; J \subset \{1,2,3\} \right\} = \Span\{X(1,2|1,2) + X(2,3|2,3) + X(1,3|1,3)\} \\
&= \{x_{11}x_{22} -x_{21}x_{12} + x_{22}x_{33}-x_{23}x_{32} + x_{11}x_{33} - x_{13}x_{31}\} \\
V_{0,3} &= \Span\left\{\sum_{|J|=3} X(J|J) \; \big| \; J \subseteq \{1,2,3\}\right\} = \{\det(X)\} \\
V_{1,2} &= \Span\left\{\sum_{|J|=1} X(P,J|Q,J) \; \big| \; P,Q \subset \{1,2,3\}, |P|=|Q|=1, (P \cup Q) \cap J = \emptyset\right\} \\
&= \Span\{X(2,1|2,1) + X(3,1|3,1) + X(2,1|3,1) + X(3,1|2,1), \\
&\;\;\;\;\;\;\;\;\;\;\;\;\;\;X(1,2|1,2)+X(3,2|3,2) + X(1,2|3,2) + X(3,2|1,2), \\
&\;\;\;\;\;\;\;\;\;\;\;\;\;\;X(1,3|1,3) + X(2,3|2,3) + X(1,3|2,3) + X(2,3|1,3)\} \\
V_{2,2} &= \Span\left\{X(P|Q) \; \big| \; P,Q \subset \{1,2,3\}, |P|=|Q|=2\right\} \\
&=\Span\{X(1,2|1,2),X(1,3|1,3),X(2,3|2,3),X(1,2|1,3),X(1,2|2,3),X(1,3|2,3),X(1,3|1,2),X(2,3|1,3),\\
&\;\;\;\;\;\;\;\;\;\;\;\;\;\;X(2,3|1,2),X(1,3|1,2),X(2,3|1,2),X(2,3|1,3),X(1,2|1,3),X(1,3|2,3),X(1,2|2,3)\}
\end{align*}
It is difficult to find reductions in the span of a system of equations as opposed to the direct computation provided by Algorithm 1. Thus, linear hulls of subsets of 21 polynomial equations generate $J_{[2,1]}$.

\section{Computing Nilpotent Orbits in $\mathfrak{sp}_{2m}$}
A symplectic matrix is a $2m \times 2m$ matrix $M$ with entries from $\mathbb{K}$ which satisfies $M^{T} \Omega M = \Omega$, where $\Omega$ is a fixed $2m \times 2m$ invertible (nonsingular) and skew-symmetric ($M^{T} = -M$) matrix, where typically
$$\Omega = \left[\begin{array}{ccccccc}
0 & 0 & 0 & 0 & 0 & 1 \\
0 & 0 & 0 & 0 &-1 & 0 \\   
  &   &\vdots & & &   \\
0 & 1 & 0 & 0 & 0 & 0 \\
-1& 0 & 0 & 0 & 0 & 0
\end{array}\right]$$
The symplectic group of degree $2m$ over a field $\mathbb{K}$ is denoted by $Sp(2m,\mathbb{K})$ and is the group of all symplectic matrices with matrix multiplication as the group operation. The symplectic Lie algebra $\mathfrak{sp}_{2m}$ is the Lie algebra of the Lie group $Sp(2m,\mathbb{K})$ and is the set of all matrices $M$ such that $e^{tM} \in Sp(2m,\mathbb{K})$. Equivalently, $\mathfrak{sp}_{2m}$ can be thought of as the tangent space to $Sp(2m,\mathbb{K})$ at the identity. We now want to compute nilpotent orbit varieties in $\mathfrak{sp}_{2m}$, which can be indexed by partitions of $2m$ for which each odd integer appears with even multiplicity due to a theorem of Gerstenhaber presented in Section 5.1 of \cite{colling}.

As lie algebras, we have $\mathfrak{sp}_{2m}$ is a subalgebra of $\mathfrak{gl}_{2m}$ and as such we can consider intersections of nilpotent orbits $\mathcal{O}_{\lambda}$ in $\mathfrak{gl}_{2m}$ with nilpotent orbits $\mathcal{O}_{\lambda}^{\mathfrak{sp}}$ in $\mathfrak{sp}_{2m}$ occurring inside the nilpotent cone $\mathcal{N}(2m)$. We now characterize the conditions of nilpotency in symplectic lie algebras by requiring the symplectic condition $X^{T} \Omega + \Omega X = 0$ along with a partition for which Gerstenhaber's theorem holds. Consider an arbitrary integer partition $\lambda = [\lambda_{1},...,\lambda_{l}]$ with $2m = \sum_{i=1}^{l} \lambda_{i}$. We have that
$$\overline{\mathcal{O}_{\lambda}} \cap \mathfrak{sp}_{2m} = \overline{\mathcal{O}_{\lambda}^{\mathfrak{sp}}}$$
and so we compute nilpotent orbit variety closures in the symplectic lie algebra $\mathfrak{sp}_{2m}$ by requiring that the symplectic condition holds:
\begin{Lemma}
Let $X \in \mathbb{K}[x_{ij} \; | \; 1 \leq i,j \leq 2m]$. Then $X$ is symplectic when $X^{T} \Omega X = \Omega$, which is when the equations in the following sets are satisfied.
\begin{align*}
\Lambda_{2m}^{\mathfrak{sp}}(2q+1,n-2q) &= \left\{1 + \sum_{k=1}^{2m} (-1)^{k} x_{2m+1-k,i}x_{k,j} = 0 \; \big| \; i=2q+1, j =n-2q, q \in \mathbb{N}, q < m \right\} \\
\Lambda_{2m}^{\mathfrak{sp}}(2q,n - 2q + 1) &= \left\{1 + \sum_{k=1}^{2m} (-1)^{k+1} x_{2m+1-k,i}x_{k,j} = 0 \; \big| \; i =2q, j = n-2q+1, q \in \mathbb{N}, q < m \right\} \\
\Lambda_{2m}^{\mathfrak{sp}}(r,s) &= \bigg\{\sum_{k=1}^{2m} (-1)^{k} x_{2m+1-k,i}x_{k,j} = 0 \; \big| \; \neg \exists q \in \mathbb{N}, (i=r=2q+1 \wedge j=s= n-2q) \\
&\vee (i=r =2q \wedge j=s =n-2q+1), 1 \leq r,s \leq 2m\bigg\}
\end{align*}
Furthermore, $|\Lambda_{2m}^{\mathfrak{sp}}(2q+1,n-2q)|=|\Lambda_{2m}^{\mathfrak{sp}}(2q,n - 2q + 1)|=m$ and $|\Lambda_{2m}^{\mathfrak{sp}}(r,s)| = 4m^2 -2m$.
\end{Lemma}
We now call $$\Lambda_{2m}^{\mathfrak{sp}} = \bigcup_{q=0}^{m-1} \Lambda_{2m}^{\mathfrak{sp}}(2q+1,n-2q) \cup \bigcup_{q=1}^{m}  \Lambda_{2m}^{\mathfrak{sp}}(2q,n - 2q + 1) \cup \bigcup_{(r,s)} \Lambda_{2m}^{\mathfrak{sp}}(r,s)$$
and note that $|\Lambda_{2m}^{\mathfrak{sp}}| = 4m^2$. We can compute nilpotent orbit varieties closures in $\mathfrak{sp}_{2m}$ with Algorithm 2.

\begin{algorithm}
\caption{$\mathfrak{sp}_{2m}$ Nilpotent Orbit Variety Closure}
\begin{algorithmic}
\REQUIRE $\lambda = [\lambda_{1},...,\lambda_{l}]$, where $\sum_{i=1}^{l} \lambda_{i} = n$ and $\lambda_{i} \in \mathbb{N}$, $\forall i \in \{1,...,l\}$.
\STATE Set $\overline{\mathcal{O}_{\lambda}} = \emptyset$.
\FORALL{$k \in \{1,...,n\}$}
\STATE Set $r = \rank(X^{k}) = \sum_{i=1}^{l} f^{k}(\lambda_{i})$
\IF{$r \geq 0$}
\FORALL{$P,Q \subset \{1,...,n\}$}
\IF{$|P|=|Q|=r+1$}
\STATE Set $\overline{\mathcal{O}_{\lambda}} = \overline{\mathcal{O}_{\lambda}} \cup \{X^{k}(P|Q) = 0\}$.
\ENDIF
\ENDFOR
\ENDIF
\ENDFOR
\STATE Set $\overline{\mathcal{O}_{\lambda}^{\mathfrak{sp}}} = \overline{\mathcal{O}_{\lambda}} \cap \Lambda_{2m}^{\mathfrak{sp}}$
\RETURN $\overline{\mathcal{O}_{\lambda}^{\mathfrak{sp}}}$
\end{algorithmic}
\end{algorithm}

For computing symplectic nilpotent orbit varieties we intersect the general linear nilpotent orbit variety with $\mathfrak{sp}_{2m}$ and obtain
$$\mathcal{O}_{\lambda}^{\mathfrak{sp}} = \bigcup_{h \in H_{\lambda}} \left(\overline{\mathcal{O}_{\lambda}^{\mathfrak{sp}}}\right)_{h}.$$

\section{N\'{e}ron Models and Future Research Directions}
Let $R$ be a Dedekind domain, that is, an integral domain in which every nonzero proper ideal factors into a product of 
prime ideals, with field of fractions $K$ and let $R_{K}$ be an abelian variety over $K$ (which is that $R_{K}$ is a 
projective algebraic variety that is also an algebraic group). A N\'{e}ron model is a universal separated smooth scheme 
$A_{R}$ over $R$ with a rational map to $A_{K}$; equivalently, N\'{e}ron models are commutative quasi-projective group 
schemes over $R$. Motivation for studying N\'{e}ron models can come from understanding good reduction of elliptic curves 
over $\mathbb{Q}$ or for understanding the Birch and Swinnerton-Dyer Conjecture which involves the Tate-Shafarevich group 
that is defined in terms of a N\'{e}ron model over $\mathbb{Z}$ for an abelian variety over $\mathbb{Q}$. For further
references regarding N\'{e}ron models, consult the seminal work \cite{neronmodel}.

We conjecture the existence of a local weak N\'{e}ron model for a nilpotent orbit variety $$\mathcal{O}_{\lambda} = \bigcup_{k=1}^{\max\{\lambda\}} \bigcup_{j=1}^{{n \choose r_{k}}^2} \Spec\left(\frac{\mathbb{K}[X,t]}{J_{\lambda}\langle h_{j,k} t - 1 \rangle}\right)$$ by considering a reduction $\mathbb{K}[X,t] \longrightarrow \mathbb{Z}_{p}[X,t]$ in the coordinate rings of each localized affine variety defined by nilpotent orbit variety closures as
$$\Spec\left(\frac{\mathbb{K}[X,t]}{J_{\lambda} \langle h_{j,k}t - 1 \rangle}\right) \longrightarrow \Spec\left(\frac{\mathbb{Z}_{p}[X,t]}{J_{\lambda} \langle h_{j,k}t - 1 \rangle}\right).$$
In order to bound the value of $p$ admissible for a given nilpotent orbit variety determined by a partition $\lambda$ of $n$, we find the maximum coefficient of the polynomials in $H_{\lambda}$ and $F_{\lambda}$ defined by
\begin{align*}
H_{\lambda} &= \bigcup_{k=1}^{\max\{\lambda\}} H_{\lambda}^{k} = \left\{h_{j,k} \in H_{\lambda}^{k} = \{X(P|Q) \neq 0 \; \big| \; P,Q \subseteq \{1,...,n\}, |P|=|Q|=r_{k}\} \; \big| \; 1 \leq j \leq {n \choose r_{k}}^2\right\} \\
F_{\lambda} &= \bigcup_{k=1}^{\max\{\lambda\}} F_{\lambda}^{k} = \left\{f_{j,k} \in F_{\lambda}^{k} = \{X^{k}(P|Q) = 0 \; \big| \; P,Q \subseteq \{1,...,n\}, |P|=|Q|=r_{k} + 1\} \; \big| \; 1 \leq j \leq {n \choose r_{k} + 1}^2\right\}
\end{align*}
We define the coefficient projection function $\pi_{r} : \mathbb{K}[X] \rightarrow \mathbb{K}$ by $\pi_{t}(g(X)) = c_{t,j,k}$, where
$$g(X) = g(x_{11},...,x_{nn}) = \sum_{t=1}^{\Omega_{g}} c_{t,j,k} \prod_{u=1}^{n} \prod_{v=1}^{n} x_{uv}^{p_{t,uv}}$$
is an arbitrary polynomial function with $c_{t,j,k} \in \mathbb{K}, p_{t,uv} \in \mathbb{N} \cup \{0\}$ and $$\Omega_{g} = \sum_{d=1}^{\deg(g)} {d + n -1 \choose n-1}$$
For indexing the variables $x_{uv}$ in the polynomial ring $\mathbb{K}[X]$, we remark that $uv$ denotes the concatenation of $u$ and $v$ as natural numbers including zero, not the product of $u$ and $v$. We now define the set of coefficients of a polynomial $g \in \mathbb{K}[X]$ by $$C_{g} = \{\pi_{t}(g(X)) \; | \; 1 \leq t \leq \Omega_{g}\}$$
and remark that the problem of determining the maximum coefficient of the polynomials in $H_{\lambda}$ and $F_{\lambda}$ is then defined by
$$p > \max\left\{\left(\bigcup_{k=1}^{\max\{\lambda\}} \bigcup_{j=1}^{{n \choose r_{k}}^2} C_{h_{j,k}} \right) \coprod \left(\bigcup_{k=1}^{\max\{\lambda\}} \bigcup_{j=1}^{{n \choose r_{k} +1}^{2}} C_{f_{j,k}}\right)\right\}$$
As such, the problem of bounding the value of $p$ in $\mathbb{Z}_{p}$ is reduced to evaluating this maximum. In order to solve this problem we present a lemma.
\begin{Lemma}\label{n-1}
Let $X$ be an $n \times n$ matrix. Then for each $i,j \in \{1,...,n\}$ there are $(n-1)!$ occurrences of $x_{ij}$ in $\det(X)$.
\end{Lemma}
\begin{proof}
We use the Leibniz formula for the determinant of an $n \times n$ matrix
$$\det(X) = \sum_{\sigma \in S_{n}} \text{sgn}(\sigma) \prod_{i=1}^{n} X_{i,\sigma(i)}$$
Let $x_{ij}$ be an arbitrary entry in $X$ and observe that for a fixed $\sigma \in S_{n}$ the entry $x_{i,\sigma(i)}$ appears exactly once in $\det(X)$. Then, since there are $(n-1)!$ permutations $\sigma \in S_{n}$ with the property that $\sigma(i) = j$ we have that $x_{ij}$ appears $(n-1)!$ times in $\det(X)$. Alternatively, since there are $n$ multiplicative terms in each additive term and $n!$ additive terms, there are $(n+1)!$ appearances of variables $x_{ij}$ for varying $i,j \in \{1,...,n\}$. Since each $x_{ij}$ appears an equal number of times in $\det(X)$ we have that each particular $x_{ij}$ occurs $(n+1)!/n^2 = (n-1)!$ times in $\det(X)$.
\end{proof}

With this fact we have the following corollary regarding embedding determinant equations in a polynomial ring with $p$-adic integer coefficients.

\begin{Corollary}
For an $n \times n$ matrix with entries $x_{ij}$ in a field $\mathbb{K}$, we have $\det(X) \in \mathbb{Z}_{p}[X]$ with $p>(n-1)!$.
\end{Corollary}
\begin{proof}
By Lemma \ref{n-1}, each $x_{ij}$ appears $(n-1)!$ times in $\det(X)$ and so there can be at most a coefficient of $(n-1)!$ for any $x_{ij}$ which implies that the image of $\det(X)$ is invariant under the map $\mathbb{K} \longrightarrow \mathbb{Z}_{p}$ with $p > (n-1)!$. Hence, $\det(X) \in \mathbb{Z}_{p}[X]$ for $p > (n-1)!$.
\end{proof}

Since the equations $h_{j,k} \in H_{\lambda}$ are expressed in terms of $r_{k} \times r_{k}$ minors and the equations in $f_{j,k} \in F_{\lambda}$ are expressed in terms of $(r_{k} + 1) \times (r_{k} + 1)$ minors, we immediately have that 
$$\max\left\{\bigcup_{k=1}^{\max\{\lambda\}} \bigcup_{j=1}^{{n \choose r_{k} +1}^{2}} C_{f_{j,k}} \right\} > \max\left\{\bigcup_{k=1}^{\max\{\lambda\}} \bigcup_{j=1}^{{n \choose r_{k}}^2} C_{h_{j,k}}\right\}$$
and that $$\max\{r_{k} \; \big| \; 1 \leq k \leq \max\{\lambda\}\}! > \max\left\{\bigcup_{k=1}^{\max\{\lambda\}} \bigcup_{j=1}^{{n \choose r_{k} +1}^{2}} C_{f_{j,k}} \right\}$$ since each $f_{j,k}$ is an $(r_{k} + 1) \times (r_{k}+1)$ determinant function with the property by Corollary 1 that it embeds in $\mathbb{Z}_{p}[X]$ with $p > (r_{k} + 1 -1)! = r_{k}!$. Therefore, we can bound the value of $p$ by
$$p > \max\{r_{k} \; \big| \; 1 \leq k \leq \max\{\lambda\}\}!$$
where $r_{k} = \rank(X^{k}) = \sum_{i=1}^{l} f^{k}(\lambda_{i})$ and
$$f(x)=
\begin{cases}
x-1 \;\;\;\; \text{if} \; x > 0 \\
0 \;\;\;\;\;\;\;\;\;\; \text{if} \; x \leq 0
\end{cases}$$
Future work will focus on the explicit construction of local weak N\'{e}ron models for nilpotent orbit varieties,
applying the Greenberg transform to these models, thus producing pro-schemes over finite fields with a remarkable property:
the set of rational points on these pro-schemes is canonically identified with the set of rational points on nilpotent
orbit varieties appearing in Lie algebras over local fields and global fields.


\begin{bibdiv}
\begin{biblist}

\bib{NT}{article}{
   author={Nakano, Daniel K.},
   author={Tanisaki, Toshiyuki},
   title={On the realization of orbit closures as support varieties},
   journal={J. Pure Appl. Algebra},
   volume={206},
   date={2006},
   number={1-2},
   pages={66--82},
   issn={0022-4049},
   review={\MR{2220082 (2007b:17029)}},
   doi={10.1016/j.jpaa.2005.04.014},
}
		
\bib{SW}{article}{
   author={Shimozono, Mark},
   author={Weyman, Jerzy},
   title={Bases for coordinate rings of conjugacy classes of nilpotent
   matrices},
   journal={J. Algebra},
   volume={220},
   date={1999},
   number={1},
   pages={1--55},
   issn={0021-8693},
   review={\MR{1713457 (2001h:20066)}},
   doi={10.1006/jabr.1999.7969},
}
	
\bib{Tanisaki}{article}{
   author={Tanisaki, Toshiyuki},
   title={Defining ideals of the closures of the conjugacy classes and
   representations of the Weyl groups},
   journal={T\^ohoku Math. J. (2)},
   volume={34},
   date={1982},
   number={4},
   pages={575--585},
   issn={0040-8735},
   review={\MR{685425 (84g:14049)}},
   doi={10.2748/tmj/1178229158},
}
	

\bib{Weyman}{article}{
   author={Weyman, J.},
   title={The equations of conjugacy classes of nilpotent matrices},
   journal={Invent. Math.},
   volume={98},
   date={1989},
   number={2},
   pages={229--245},
   issn={0020-9910},
   review={\MR{1016262 (91g:20070)}},
   doi={10.1007/BF01388851},
}
		
\bib{FH}{book}{
   author={Fulton, W.},
   author={Harris, J.},
   title={Representation theory: a first course},
   volume={129},
   year={1991},
   publisher={Springer}
}

\bib{KLMW}{book}{
   author={Kreiman, V.},
   author={Lakshmibai, V.},
   author={Magyar, P.},
   author={Weyman, J.},
   title={On ideal generators for affine Schubert varieties},
   year={2007},
   publisher={Alpha Science Int'l Ltd}
}

\bib{colling}{book}{
   author={Collingwood, D.H.},
   author={McGovern, W.M.},
   title={Nilpotent Orbits in Semisimple Lie Algebras},
   year={1993},
   publisher={Van Nostrand Reinhold}
}

\bib{neronmodel}{book}{
author={Bosch, S.},
author={L\"{u}tkebohmert, W.},
author={Raynaud, M.},
title ={N\'{e}ron Models},
year={1990},
publisher={Springer}
}

\end{biblist}
\end{bibdiv}
\end{document}